\documentclass[12pt]{amsart}
\usepackage{amsmath,amsthm,amssymb,color,enumerate}
\usepackage{graphicx}
\usepackage[normalem]{ulem}

\usepackage{hyperref}
\usepackage{amsmath,color}
\usepackage{amssymb}
\usepackage{amscd}
\usepackage{hyperref}
\usepackage{longtable}
\usepackage{graphicx}
\usepackage{comment}
\usepackage{enumitem}
\newcommand\fk[1]{\mathfrak #1}

\newcommand{\bC}{\mathbb{C}}

\newcommand{\cO}{\mathcal{O}}

\newcommand\la{\lambda}

\newcommand\ie{\textit{i.e. }}
\newcommand\restr[2]{{
  \left.\kern-\nulldelimiterspace 
  #1 
  \vphantom{\big|} 
  \right|_{#2} 
  }}

\newtheorem{theorem}{Theorem}[section]
\newtheorem*{theorem*}{Theorem}
\newtheorem{thm}{Theorem}
\newtheorem{proposition}[theorem]{Proposition}
\newtheorem*{proposition*}{Proposition}
\newtheorem{lemma}[theorem]{Lemma}
\newtheorem*{lemma*}{Lemma}
\newtheorem{corollary}[theorem]{Corollary}
\newtheorem*{corollary*}{Corollary}

\newtheorem*{definition*}{Definition}

\newtheorem*{remark*}{Remark}
\newtheorem{example}[theorem]{Example}

\usepackage{url}
\makeatletter
\g@addto@macro{\UrlBreaks}{\UrlOrds}
\makeatother

\begin{document}
\title{Finite type multiple flag varieties of exceptional groups}
\author[D. Barbasch, S. Da Silva, B. Elek, G.G. Krishnan]
{Dan Barbasch, Sergio Da Silva, Bal\'{a}zs Elek,\\
Gautam Gopal Krishnan}

\address{Dan Barbasch, Cornell University, Ithaca NY}
\email{dmb14@cornell.edu }

\address{Sergio Da Silva, Cornell University, Ithaca NY}
\email{smd322@cornell.edu}

\address{Bal\'{a}zs Elek, Cornell University, Ithaca NY}
\email{be94@cornell.edu }

\address{Gautam Gopal Krishnan, Cornell University, Ithaca NY}
\email{gk379@cornell.edu  }

\thanks{Research supported in part by an NSERC PGS-D3 Scholarship, and an
NSA grant H98230-16-1-0006}
\maketitle

\begin{abstract}
Consider a simple complex Lie group $G$ acting diagonally on a triple
flag variety $G/P_1\times G/P_2\times G/P_3$, where $P_i$ is parabolic
subgroup of $G$. We provide an algorithm for systematically checking
when this action has finitely many orbits. We then use this method to
give a complete classification for when $G$ is of type $F_4$. The
$E_6, E_7,$ and $E_8$ cases will be treated in a subsequent paper.  
\end{abstract}
{\hypersetup{linkcolor=black}\tableofcontents}

Let $G$ be a simple complex Lie group, and fix a maximal torus $T$ as
well as a Borel subgroup $B$ containing $T$. Let $P\supseteq B$ be a
parabolic subgroup. Then $G/P$ is a projective algebraic variety
called a partial flag variety. By the Bruhat decomposition, $B$ acts
on $G/P$ with finitely many orbits indexed by $W/W_P$, where 
$W=N_G(T)/T$ is the Weyl group of $G$ and $W_P$ is the Weyl group of
$P$. The closures of these orbits $X^w:=\overline{BwP/P}$ are called
Schubert varieties. They play an important role in the 
interplay between geometry and representation theory.  An
equivalent formulation is to study $G$-orbits on $G/B\times G/B$.

A number of constructions can be made canonical by identifying $W$ with
the $G$-diagonal orbits in $G/B\times G/B$ (see \cite{CG}, Theorem
3.1.9. for example). In this context, a Schubert variety can be
thought of as the closure of one of the finitely many $G$-orbits in
$G/B\times G/B$ . One can then ask about the natural extension of this
construction: when does the diagonal $G$-action on $G/P_1\times \ldots
\times G/P_k$ have finitely many orbits?  

In \cite{MWZ1}, Magyar, Weyman and Zelevinsky give a complete answer
to this question for $G=GL_n(\bC)$ while also providing
representatives for the orbits in each case. A key aspect in their
approach is identifying $G/P$ with a partial flag variety 
$Fl_P=\{ \{0\}\subset V_1\subset \ldots \subset V_k \subset \bC^n
\}$. Using a similar method, they extend these results to the type $C$
case in \cite{MWZ2}. In both instances, this partial flag variety
interpretation is used in an essential way to obtain the
classification results. 

The case $P_1=\ldots =P_k=P$ has been studied by Popov \cite{P} and
Devyatov \cite{D}, where they also consider the question when a
variety $(G/P)^k$ has an open orbit under the diagonal action of $G$. 

Littelmann, in \cite{L}, considers the case when such a multiple flag
variety is spherical, that is, when it has an open orbit under the
action of $B$. This is naturally equivalent to choosing $P_1=B$ in our
setup. Also note that by the results of Brion \cite{B}, spherical
varieties have only finitely many $B$-orbits.  

The references \cite{AP},\cite{HNOO}, \cite{L} and \cite{S} deal with
aspects of the geometry of these orbits and representation theory.
The methods employed in each of these articles do not readily extend
to the exceptional groups. We had to therefore take a more
direct approach for computing the orbits of the diagonal $G$-action in
question. We designed an algorithm, detailed in Section
\ref{algorithm}, which leads to a complete classification summarized
in the theorem below. It is a synthesis of Propositions
\ref{infinitemax}, \ref{finitemax}, \ref{nonmax2} and Corollary
\ref{nonmax1}. 

\begin{thm}
If $G$ is of type $F_4$, then the diagonal $G$-action on  $G/P_1\times
G/P_2\times G/P_3$ has finitely many orbits iff
$(P_1,P_2,P_3)=(B_3,C_3,P_{max})$ up to permutation. Here $P_{max}$ is
any maximal parabolic subgroup of $G$. 
\end{thm}

Section \ref{dimension} outlines an easy dimension count which shows
that $G$ cannot have an open orbit on products of four or more flag
varieties. 

Section \ref{algorithm} details the method we
use. The general idea is as follows: First, any orbit has a
representative $(eP_1,xP_2,yP_3)$. The stabilizer of the first
coordinate is $P_1,$ and this group acts on the remaining two coordinates;
each orbit has a representative with ${x=w\in W(P_1)\backslash W/W(P_2)}.$ Finally
the stabilizer of the first two coordinates, $S_{12}(w)\subset P_1$ acts on the third
coordinate. Each of the cells $P_1vP_3$ with ${v\in W(P_1)\backslash
W/W(P_3)}$ is stabilized by $S_{12}(w),$ 
and it remains to compute the orbits. This reduces to a calculation of
the orbits of $S_{12}$ on a vector space $V_{13}$. The group has a Levi
decomposition $S_{12}(w)=M_{12}(w)N_{12}(w)$. 

Then $V_{13}$ has a natural grading whose associated filtration is compatible with that of
$S_{12}(w).$ The action is such that $M_{12}(w)$ preserves
the grading while $N_{12}(w)$ moves it \textit{strictly upwards}. The
cases when there are finitely many orbits of $M_{12}(w)$ on each level
are determined from existing literature. The action of $N_{12}(w)$ is determined using a computer program.

Section \ref{compute} then presents sample calculations
from the algorithm. The computer code used to execute most of the
algorithm can be found at the following link:\\  

\begin{sloppypar}
\url{https://cocalc.com/projects/4bbd59fb-232c-4eeb-8b7a-f52d36ef3351/files/FTMF.sagews}
\end{sloppypar}

\section{Using a Dimension Count to Eliminate Cases }\label{dimension}
Let us first observe that if $G$ acts on \mbox{$G/P_1\times \ldots\times G/P_k$} with finitely many
  orbits, then there must be an open orbit and necessarily 
\[\dim(G)\geq \dim(G/P_1\times \ldots \times G/P_k) = \sum_{i=1}^k \dim(G/P_i) .\]

This already lets us eliminate a number of cases.
  For instance, in the fairly trivial case that $G$ is of type $G_2$, we find that $k\leq 2$, and the question is completely answered by the Bruhat decomposition. 

Let us consider the case when $G$ is of type $F_4$. Fix a Borel
subgroup $B$ and a Cartan subgroup $H\subset B.$ Then
$\dim(G)=52$. Parabolic subgroups containing $B$  are indexed by subsets of the
Dynkin diagram for $F_4$. This leads to the following list of
parabolic subgroups and their dimensions: 

\[ (\emptyset,28), (A_1, 29), (A_1\times A_1, 30),(A_2, 31), (B_2, 32),\]
\[(A_2\times A_1,32), (B_3, 37), (C_3, 37)\]

The $A_i, B_i, C_i$ refer to the Dynkin diagrams of the
Levi components. At this stage we do not distinguish between long and short roots in
the list since this does not affect dimension. The possibilities for
the dimension of $G/P$ are therefore:

\[24,23,22,21,20,15.\]

Since $\dim(G/P_1\times \ldots \times G/P_k) \geq 60$ for any $k\geq
4$, we only need to consider $k=2,3$. 
The case $k=2$ is already done by the Bruhat decomposition, which
leaves only triples of parabolic subgroups to classify. The only triples $(P_1,P_2,P_3)$ of parabolic subgroups that need to be considered (up to permutation) are those for which

\[ \dim(G/P_1)=\dim(G/P_2)=15 \text{ , } \dim(G/P_3)= 15,20,21 \text{ or } 22.\]

\section{An Algorithm for Computing Orbits}\label{algorithm}

Let $G$ be a connected simple linear algebraic group acting diagonally
on the variety $G/P_1\times G/P_2\times G/P_3$. Recall the fixed pair $(B,H)$ of
a Borel subgroup and Cartan subgroup, and assume that the $P_i$ are
standard, \ie $B\subset
P_i$. Denote by $\fk g, \fk p_i,\fk b, \fk h$ and so on the corresponding Lie
algebras. For each $P_i$ we choose an element $\la_i\in\fk h$ that
determines it, in the sense that $P_i=L_iN_i$ where $H\subset L_i$ is
a Levi component, and $N_i$ the unipotent radical so that the roots in
$L_i$ are zero, and the roots in $N_i$ are positive on
$\la_i.$

For each orbit we determine a representative by the steps that follow. This representative is unique, so there are finitely many
orbits if and only if we obtain finitely many representatives. We call these \textit{distinguished representatives}

\begin{enumerate}

\item 
{The action of $G$ on $G/P_i$ is transitive; there is only one
  orbit $P_i.$ The diagonal action of the group $G$ on a product of two flag
  varieties, either 
  $G/P_1\times G/P_2$ or  $G/P_1\times G/P_3$, has finitely many
  orbits which are given by the Bruhat decomposition. For every orbit
  we choose the coset representative $P_1\times gP_i.$ The stabilizer
   is $P_1,$ and it acts on 
    $G/P_2\times G/P_3$. The generalized Bruhat decompositions are   $G=\cup P_1wP_2$
and 
  $G=\cup P_1vP_3$ with  $w\in
  W_{P_1}\backslash W/W_{P_2}$ and $v\in W_{P_1}\backslash W/W_{P_3}$.
 Therefore each orbit has a representative  of the form
 $\{P_1\}\times \{p_1wP_2\}\times \{q_1vP_3\}$ with $p_1,q_1\in P_1$. We refer to the representative in $G/P_i$ as
the $i$-th coordinate. The
   stabilizer of $\{P_1\}$, which is $P_1$, acts on the second and
   third coordinate diagonally. For the orbits of
   $P_1$ in the second
   coordinate we choose the cosets $wP_2$; so each orbit has a
   representative $\{P_1\}\times\{wP_2\}\times \{q_1 vP_3\}$. The stabilizer
in $P_1$ of the first two \textit{coordinates}
is $S_{12}(w):=P_1\cap wP_2w^{-1}$. This stabilizer acts on
$P_1vP_3\subset G/P_3$  with 
finitely many orbits if and only if  $S_{12}(w)$ acts with finitely many
orbits on $\cO(v):=P_1vP_3$ for all 
\[
(w,v)\in  W_{P_1}\backslash W/W_{P_2}\times W_{P_1}\backslash W/W_{P_3}.
\]

\item 
Write  $P_1=M_1N_1$, the Levi decomposition corresponding to $\la_1.$  
The group $S_{12}(w)\subset P_1$ inherits a decomposition
\[S_{12}(w)=M_{12}(w)N_{12}(w).\]
Similarly write 
\[S_{13}(v)=P_1\cap vP_3v^{-1}=M_{13}(v)N_{13}(v).\] 
$M_{12}(w)$ and $M_{13}(v)$ are (standard) parabolic subgroups of
$M_1.$ We write the Levi decompositions as $M_{12}(w)=M_{12}(w)^0M_{12}(w)^+$ and 
${M_{13}(v)=M_{13}(v)^0M_{13}(v)^+}$ corresponding to $w\la_2$
and $v\la_3$ respectively. Then 
\[\indent\indent
P_1vP_3\cong \big[P_1/S_{13}(v)\big] vP_3\cong \big[M_1/M_{13}(v)\times
N_1/N_{13}(v)\big] vP_3.
\]  
\item The action of $S_{12}(w)$ is compatible with the above
  description. 
Since $M_{12}(w),M_{13}(v)\subset M_1$ are parabolic subgroups,
 $M_1$ has a Bruhat decomposition $M_1=\bigcup M_{12}(w)xM_{13}(v)$ with 
$x\in W_{M_{12}(w)}\backslash W_{M_1} / W_{M_{13}(v)}$. Using the
action of $M_{12}(w)$ (which is multiplication on the left)  we can
choose  representatives for the orbits $xnvP_3$
with $x$ as above, and $n\in N_1/N_{13}(v).$ 

Elements $n_{12}\in N_{12}$ satisfy $x^{-1}n_{12}x\in N_1$, so 
\[
n_{12}xnvP_3=x(x^{-1}n_{12}x)nvP_3,
\]
and we can take representatives of the orbits  in 
\[
x\big[x^{-1}N_{12}(w)x\backslash N_1/N_{13}(v)\big] vP_3.
\]
Elements $m=xm_{13}x^{-1}\in M_{12}(w)\cap xM_{13}(v)x^{-1}$ satisfy
\[\indent\indent\indent
xm_{13}x^{-1}xnvP_3=x(m_{13}nm_{13}^{-1})v(v^{-1}m_{13}v)P_3=x(m_{13}nm_{13}^{-1})vP_3.
\]
In conclusion, orbits of $G$ acting diagonally on $G/P_1\times G/P_2\times G/P_3$ have
representatives of the form
\[
P_1\times wxP_2\times qvP_3
\]
where $q\in\big(x^{-1}N_{12}(w)x\big)\backslash N_1/N_{13}(v) $ are
orbit representatives under  the adjoint action of
$x^{-1}M_{12}(w)x\cap M_{13}(v)$.\newline

\item 
Since $N_1$ is nilpotent, we can carry out the
  calculations in the Lie algebra. We use the filtrations coming from levels of $\la_i$ and their respective conjugates under $v,w$ and $x$ to keep track of
  actions of nilpotent groups. The (known) lists of
  representations of reductive groups on their irreducible
  representations with finitely many orbits come into play. 
}

\end{enumerate}

\section{Sample Computations}\label{compute}

Let us begin with a lemma which will reduce the number of cases we need to check.

\begin{lemma}\label{lemma}
	If the diagonal $G$-action on $G/P_1\times G/P_2\times G/P_3$ has infinitely many orbits, then so does the diagonal $G$-action on $G/Q_1\times G/Q_2\times G/Q_3$ for $Q_i \subset P_i$. Furthermore, if the second action has an open orbit, then so does the first.
\end{lemma}
\begin{proof}
	The map $G/Q_1\times G/Q_2\times G/Q_3$ onto $G/P_1\times G/P_2\times G/P_3$ is $G$-equivariant. Therefore, the fibre over an orbit is the union of orbits. By surjectivity, the first statement follows.
    
    The preimage of an orbit $\mathcal{O}_P$ in $G/P_1\times G/P_2\times G/P_3$ is a union of orbits (by $G$-equivariance) which all have the same dimension (since fibres of this quotient map are equidimensional). Suppose that the fibre over $\mathcal{O}_P$ contains some open orbit $\mathcal{O}_Q$.  Since open orbits are dense (and hence unique), $\mathcal{O}_Q$ must be the full preimage of $\mathcal{O}_P$. Therefore, $\mathcal{O}_P$ is open by the quotient topology. 
\end{proof}

We start by checking the maximal cases. When they have infinitely many orbits, so do the smaller parabolic cases by Lemma \ref{lemma}. We cannot however conclude anything about these smaller cases based on a finite maximal case. These have to be checked separately (see Section \ref{finite}).
We begin with the following maximal choices for $(P_1,P_2,P_3)$ not eliminated by the dimension count:
\[(B_3,B_3,B_3),(B_3,B_3,C_3),(C_3,C_3,B_3), (B_3,B_3,A_1(l)\times A_2(s)),\] \[(B_3,B_3,A_2(l)\times A_1(s)), (B_3,C_3,A_1(l)\times A_2(s)),(B_3,C_3,A_2(l)\times A_1(s))\] \[(C_3,C_3,C_3),(C_3,C_3, A_1(l)\times A_2(s)),(C_3,C_3, A_2(l)\times A_1(s)).\]

To fix the notation for the computations that follow, label the Dynkin diagram for $F_4$ from the long roots to the short roots using the simple roots:
\[\varepsilon_2 - \varepsilon_3\text{ , } \varepsilon_3 - \varepsilon_4 \text{ , } \varepsilon_4\text{ , } \frac{1}{2}(\varepsilon_1 - \varepsilon_2 - \varepsilon_3 - \varepsilon_4)\]
where $\rho = (\frac{11}{2},\frac{5}{2},\frac{3}{2},\frac{1}{2})$  is the sum of the fundamental weights. We distinguish between $A_2(l)\times A_1(s)$ (with $A_2$ corresponding to the long roots) and $A_1(l)\times A_2(s)$ (with $A_2$ given by the short roots).

\subsection{Infinitely many orbits}\label{infinite}

The triple $B_3$ case and the triple $C_3$ case are shown to have infinitely many orbits in \cite{D}. This leaves 8 more triples to check. Four of these remaining eight cases have infinitely many orbits:

\begin{proposition}\label{infinitemax}
The following maximal parabolic cases 
\[(B_3,B_3,B_3), (B_3,B_3,A_1(l)\times A_2(s)),(B_3,B_3,A_2(l)\times A_1(s)),\] \[(C_3,C_3,C_3),(C_3,C_3, A_1(l)\times A_2(s)),(C_3,C_3, A_2(l)\times A_1(s)),\] 
have infinitely many orbits.
\end{proposition}
\begin{proof}
The algorithm from Section \ref{algorithm} gives infinitely many
orbits for  $v$ and $w$ the longest minimal coset representatives and $x$ the identity.
\end{proof}

We now provide a detailed example of how the algorithm is used to show that there are infinitely many orbits. One case is provided; the remaining ones are left to the reader.

\begin{example}\label{ex1}
\normalfont Consider the $(B_3,B_3,A_1(l)\times A_2(s))$ case. We can form a filtration of the roots by taking their dot product against $\gamma = (3,1,1,0)$:\\ 

\underline{$x^{-1}M_{12}(w)x\cap M_{13}(v)$:} \newline

\begin{itemize}[label={}]
\item $\mathcal{M}_0: \pm\varepsilon_4, \pm(\varepsilon_2 - \varepsilon_3), \mathfrak{h}$
\item $\mathcal{M}_1: \varepsilon_2,\varepsilon_3, \varepsilon_2 \pm \varepsilon_4, \varepsilon_3 \pm \varepsilon_4$
\item $\mathcal{M}_2: \varepsilon_2 +\varepsilon_3$\newline
\end{itemize}

\underline{$x^{-1}N_{12}(w)x\backslash N_1/N_{13}(v)$:}\newline

\begin{itemize}[label={}]
\item $\mathcal{N}_{\frac{1}{2}}:\frac{1}{2}(\varepsilon_1 -\varepsilon_2 - \varepsilon_3 \pm \varepsilon_4)$
\item $\mathcal{N}_{\frac{3}{2}}: \frac{1}{2}(\varepsilon_1 +\varepsilon_2 - \varepsilon_3 \pm \varepsilon_4), \frac{1}{2}(\varepsilon_1 -\varepsilon_2 + \varepsilon_3 \pm \varepsilon_4)$
\item $\mathcal{N}_{2}:\varepsilon_1 - \varepsilon_2 , \varepsilon_1 - \varepsilon_3$
\item $\mathcal{N}_{3}:\varepsilon_1, \varepsilon_1 \pm \varepsilon_4$
\item $\mathcal{N}_{4}:\varepsilon_1 + \varepsilon_2 , \varepsilon_1 + \varepsilon_3 $\newline
\end{itemize}

We restrict our attention to orbits with representatives of the form:
\[\alpha_1x(\frac{1}{2}(\varepsilon_1 +\varepsilon_2 - \varepsilon_3 + \varepsilon_4)) + \alpha_2x(\frac{1}{2}(\varepsilon_1 +\varepsilon_2 - \varepsilon_3 - \varepsilon_4)) +\]
\[ \alpha_3x(\frac{1}{2}(\varepsilon_1 -\varepsilon_2 + \varepsilon_3 + \varepsilon_4)) + \alpha_4x(\frac{1}{2}(\varepsilon_1 -\varepsilon_2 + \varepsilon_3 - \varepsilon_4)) +  \]
\[\alpha_5x(\varepsilon_1) +\alpha_6x(\varepsilon_1 + \varepsilon_4)+ \alpha_7x(\varepsilon_1 - \varepsilon_4)+\alpha_8x(\varepsilon_1 + \varepsilon_2)+\alpha_9x(\varepsilon_1 + \varepsilon_3)\]
where $\alpha_i\neq 0$ are complex coefficients and $x(\beta)$ is a root vector corresponding to the root $\beta$. We are only using $\beta$ from $\mathcal{N}_{\frac{3}{2}},\mathcal{N}_3,\mathcal{N}_4$.

Observe that we can: 
\begin{itemize}
\item set $\alpha_2 = 0$ using the action of $x(-\varepsilon_4)$ on $\alpha_1x(\frac{1}{2}(\varepsilon_1 +\varepsilon_2 - \varepsilon_3 + \varepsilon_4))$ 
\item set $\alpha_3 = 0$ using the action of $x(\varepsilon_4)$ on $\alpha_4x(\frac{1}{2}(\varepsilon_1 -\varepsilon_2 + \varepsilon_3 - \varepsilon_4))$
\item set $\alpha_4=0$ with $x(\varepsilon_3-\varepsilon_2)$ acting on $\alpha_1x(\frac{1}{2}(\varepsilon_1 +\varepsilon_2 - \varepsilon_3 + \varepsilon_4))$
\end{itemize}

There was a choice in how we used each root in $\mathcal{M}_0$. However, in each possibility we can arrange for exactly three of the $\alpha_1,\alpha_2,\alpha_3$ or $\alpha_4$ to be zero. This choice will not affect the final conclusion.

We now notice that the action of $\mathcal{M}_1$ and $\mathcal{M}_2$ cannot eliminate $\alpha_5,\alpha_6$ or $\alpha_7$. Therefore, we have shown that any further restriction of the representative above will contain roots which are linearly dependent. Hence the Cartan subalgebra cannot be used to make all remaining coefficients equal to $1$, and we have infinitely many orbits as a result. 
\end{example}

\begin{corollary}\label{nonmax1}
The following list of non-maximal parabolic cases 
\[(B_3,B_3,B_2),(B_3,B_3,A_2), (B_3,B_3, A_1\times A_1),\]
\[(C_3,C_3,B_2),(C_3,C_3,A_2),(C_3,C_3, A_1\times A_1),\]
have infinitely many orbits for any choice of $A_2$ and $A_1\times A_1$.
\end{corollary}
\begin{proof}
This follows from Proposition \ref{infinitemax} and Lemma \ref{lemma}
\end{proof}

\subsection{Finitely many  orbits}\label{finite}

While only one choice of $(v,w,x)$ was needed to show $(P_1,P_2,P_3)$ had infinitely many orbits, the finite case requires far more verification. The four remaining maximal cases do indeed have finitely many orbits. 

\begin{proposition}\label{finitemax}
The following triples of parabolic subgroups
\[(B_3,C_3, A_1(l)\times A_2(s)),(B_3,C_3,B_3),\]\[ (B_3,C_3, A_2(l)\times A_1(s)), (B_3,C_3,C_3),\]
are the only maximal cases with finitely many orbits.
\end{proposition}
\begin{proof}
To show this, we have verified that every choice of triple $(v,w,x)$
yields finitely many orbits. We omit the details, but provide one example
below. By Proposition \ref{infinitemax}, all other
parabolic cases have infinitely many orbits.
\end{proof}

In the $(B_3,C_3,A_2(l)\times A_1(s))$ case for example, there are 73 of these $(v,w,x)$ triples to check. They are not difficult to verify, so we leave it up the reader to apply the algorithm to any subcase that is of interest. The next example is dedicated to computing one of these subcases.

\begin{example}
\normalfont Consider the $(B_3,C_3,A_2(l)\times A_1(s))$ case. Let $w$ and $v$ be the elements of the Weyl group such that $w$ conjugates $\rho$ into $(-\frac{5}{2},\frac{11}{2},\frac{3}{2},\frac{1}{2})$ and $v$ conjugates $\rho$ into $(\frac{1}{2},\frac{9}{2},\frac{7}{2},\frac{5}{2})$. For this example, $x$ is chosen to be the identity. Filter the roots by taking their dot product against $\gamma = (\frac{7}{2},\frac{3}{2},\frac{1}{2},\frac{1}{2})$: 
\newline

\underline{$x^{-1}M_{12}(w)x\cap M_{13}(v)$:}\newline

\begin{itemize}[label={}]
\item $\mathcal{M}_0: \pm(\varepsilon_3 - \varepsilon_4), \mathfrak{h}$
\item $\mathcal{M}_{\frac{1}{2}}: \varepsilon_3, \varepsilon_4$
\item $\mathcal{M}_1: \varepsilon_3 + \varepsilon_4, \varepsilon_2 - \varepsilon_3, \varepsilon_2 - \varepsilon_4$
\item $\mathcal{M}_{\frac{3}{2}}: \varepsilon_2$
\item $\mathcal{M}_2: \varepsilon_2 + \varepsilon_3, \varepsilon_2 + \varepsilon_4$\\
\end{itemize}

\underline{$x^{-1}N_{12}(w)x\backslash N_1/N_{13}(v)$:}\newline

\begin{itemize}[label={}]
\item $\mathcal{N}_{\frac{1}{2}}: \frac{1}{2}(\varepsilon_1 -\varepsilon_2 - \varepsilon_3 - \varepsilon_4)$
\item $\mathcal{N}_1: \frac{1}{2}(\varepsilon_1 -\varepsilon_2 + \varepsilon_3 - \varepsilon_4), \frac{1}{2}(\varepsilon_1 -\varepsilon_2 - \varepsilon_3 + \varepsilon_4)$
\item $\mathcal{N}_2: \varepsilon_1 - \varepsilon_2$
\item $\mathcal{N}_3: \varepsilon_1 - \varepsilon_3, \varepsilon_1 - \varepsilon_4$\newline
\end{itemize}

A generic representative of an orbit is of the form:

\[ \alpha_1x(\frac{1}{2}(\varepsilon_1 -\varepsilon_2 - \varepsilon_3 - \varepsilon_4)) + \alpha_2x(\frac{1}{2}(\varepsilon_1 -\varepsilon_2 + \varepsilon_3 - \varepsilon_4)) +\]
\[ \alpha_3x(\frac{1}{2}(\varepsilon_1 -\varepsilon_2 - \varepsilon_3 + \varepsilon_4)) + \alpha_4x(\varepsilon_1 - \varepsilon_2) + \alpha_5x(\varepsilon_1 - \varepsilon_3) +\alpha_6x(\varepsilon_1 - \varepsilon_4)\]\\
where $\alpha_i$ are complex coefficients and $x(\beta)$ is a root vector corresponding to the root $\beta$.

Observe that if $\alpha_1 \ne 0$, we can: 
\begin{itemize}
\item set $\alpha_2 = 0$ using the action of $x(\varepsilon_3)$ on $\alpha_1x(\frac{1}{2}(\varepsilon_1 -\varepsilon_2 - \varepsilon_3 - \varepsilon_4))$ 
\item set $\alpha_3 = 0$ using the action of $x(\varepsilon_4)$ on $\alpha_1x(\frac{1}{2}(\varepsilon_1 -\varepsilon_2 - \varepsilon_3 - \varepsilon_4)$
\end{itemize}
On the other hand if $\alpha_1 = 0$, we can set at least one of $\alpha_2$ or $\alpha_3$ to $0$ by using the reductive part of $x^{-1}M_{12}(w)x\cap M_{13}(v)$ (i.e. part of $\mathcal{M}_0$).\newline

Similarly, if $\alpha_4 \ne 0$, we can
\begin{itemize}
\item set $\alpha_5 = 0$ using the action of $x(\varepsilon_2 - \varepsilon_3)$ on $\alpha_4x(\varepsilon_1 - \varepsilon_2)$
\item set $\alpha_6 = 0$ using the action of $x(\varepsilon_3 - \varepsilon_4)$ on $\alpha_4x(\varepsilon_1 - \varepsilon_2)$ 
\end{itemize}
Notice that we can do this without affecting our previous reductions because $x(\varepsilon_2 - \varepsilon_3)$ and $x(\varepsilon_3 - \varepsilon_4)$ centralize the roots in $\mathcal{N}_{1/2}$ and $\mathcal{N}_1$.

Finally, if $\alpha_4 = 0$, then we can set at least one of $\alpha_5$ or $\alpha_6$ to $0$ by utilizing the roots in the reductive part of $x^{-1}M_{12}(w)x\cap M_{13}(v)$ that have not already been used.

Therefore, we have shown that any orbit has a representative with at most two non-zero coefficients $\alpha_i$. This means that the roots that appear in this representative are linearly independent. Hence the Cartan subalgebra can be used to make these coefficients $1$. There are in fact $16$ orbits for this particular choice of $(w,v,x)$.
\end{example}

At this point, the only unclassified non-maximal cases are 
the ones discussed in the next proposition:

\begin{proposition}\label{nonmax2}
The following non-maximal cases 
\[(B_3,C_3,B_2),(B_3,C_3,A_2),(B_3,C_3, A_1\times A_1),\]
have infinitely many orbits for any choice of $A_2$ and $A_1\times A_1$.
\end{proposition}
\begin{proof}
  Lemma \ref{lemma} cannot be applied in this situation, but one can
  check by direct computation that they all have infinitely many
  orbits. 
\end{proof}
\subsection{Cases with open orbits}

The existence of an open orbit for the $G$-diagonal action on
$G/P_1\times G/P_2\times G/P_3$ is important when considering
compactification questions. When  there are finitely many orbits,  an open orbit
necessarily exists. Any case for which there are infinitely many orbits eliminated using the
dimension count (from Section \ref{dimension}) cannot have open
orbits. The remaining cases with an open orbit are listed in the proposition below. This is consistent with
Lemma \ref{lemma}. 

\begin{proposition}
The following infinite orbit cases 
\[ (B_3,C_3,B_2), (B_3,B_3,A_2(s)),(B_3,B_3,A_1(l)\times A_2(s)),\] \[(C_3,C_3,A_2(l)),(C_3,C_3,A_2(l)\times A_1(s)),\]
have an open orbit.
\end{proposition}
\begin{proof}
We directly computed case by case that there are finitely many orbits when
${x^{-1}M_{12}(w)x\cap M_{13}(v)}$ acts on a generic representative (all coefficients nonzero) of $x^{-1}N_{12}(w)x\backslash N_1/N_{13}(v)$, where $v,w,x$ are the longest choices of respective coset representatives. 
\end{proof}

\end{document}